\def\version{Version 5 -- last updated 19/11/2024
\hfill\href{https://arxiv.org/abs/2309.05921}{arXiv:2309.05921}
}

\documentclass[11pt,a4paper]{amsart}

\usepackage{fullpage}

\usepackage[hypertexnames=false]{hyperref}

\usepackage{amssymb,amscd,amsxtra}
\usepackage{autobreak}

\usepackage[alphabetic,lite]{amsrefs}

\usepackage[autobold]{mathfixs}

\usepackage[scr,scaled=1.1]{rsfso}

\usepackage{leftidx}

\usepackage[clockwise]{rotating}

\usepackage{xkvltxp}
\usepackage[nomargin,inline]{fixme}
\fxusetheme{color}
\definecolor{fxnote}{rgb}{0.0000,0.6000,0.0000}
\usepackage[]{todonotes}

\usepackage{color-edits}
\addauthor[AB]{AB}{magenta}

\usepackage{mathtools}
\usepackage{relsize}

\usepackage{tikz}
\usetikzlibrary{arrows,calc}

\usepackage{tikz-cd}
\usepackage{tikz-network}

\usepackage{braket}

\usepackage{pigpen}

\newtheorem{thm}{Theorem}[section]

\newtheorem{lem}[thm]{Lemma}
\newtheorem{prop}[thm]{Proposition}
\newtheorem{cor}[thm]{Corollary}

\theoremstyle{definition}

\newtheorem*{rem*}{Remark}

\newtheorem{examp}[thm]{Example}

\numberwithin{equation}{section}
\numberwithin{figure}{section}

\def\ds{\displaystyle}
\def\:{\colon}
\def\.{\cdot}

\def\<{\left\langle}
\def\>{\right\rangle}
\def\({\left(}
\def\){\right)}
\def\ph#1{\phantom{#1}}
\def\epsilon{\varepsilon}
\def\phi{\varphi}

\def\leq{\leqslant}
\def\geq{\geqslant}

\def\bar#1{\overline{#1}}

\def\tilde#1{\widetilde{#1}}
\def\iso{\cong}

\def\F{\mathbb{F}}

\def\k{\Bbbk}

\def\Z{\mathbb{Z}}

\DeclareMathOperator{\Coext}{Coext}
\DeclareMathOperator{\Cohom}{Cohom}

\DeclareMathOperator{\cone}{C}

\DeclareMathOperator{\End}{End}
\DeclareMathOperator{\Ext}{Ext}
\DeclareMathOperator{\Gal}{Gal}

\DeclareMathOperator{\Hom}{Hom}

\DeclareMathOperator{\Pic}{Pic}

\def\SO{\mathrm{SO}}

\def\tmf{{\mathrm{tmf}}}

\DeclareMathOperator{\Sq}{Sq}

\def\dlQ{\mathrm{Q}}

\def\QS0{\dlQ S^0}
\def\QSo0{\dlQ_0S^0}
\def\StA{\mathcal{A}}

\def\StE{\mathcal{E}}

\DeclareMathOperator{\rad}{rad}

\def\op{\mathrm{o}}

\title
[Endotrivial modules and iterated Jokers
in chromatic homotopy theory]
{Endotrivial modules for the quaternion
group and iterated Jokers in chromatic
homotopy theory}
\author{Andrew Baker}
\date{\version}
\address{
School of Mathematics \& Statistics,
University of Glasgow, Glasgow G12~8QQ, Scotland.}
\email{andrew.j.baker@glasgow.ac.uk}
\urladdr{http://www.maths.gla.ac.uk/$\sim$ajb}
\thanks{I would like to thank the following
for helpful comments: Dave Benson, Ken Brown,
Bob Bruner, Hans-Werner Henn, Lennart Meier,
Doug Ravenel, John Rognes, Danny Shi, and
Vesna Stojanoska. I would like to acknowledge
the support of LAGA, l'Universit\'e Sorbonne,
Paris Nord where an early version of this 
paper was completed.}

\begin{document}

\begin{abstract}
The algebraic Joker module was originally
described in the 1970s by Adams and Priddy
and is a $5$-dimensional module over the
subHopf algebra $\StA(1)$ of the mod~$2$
Steenrod algebra. It is a self-dual
\emph{endotrivial module}, i.e., an
invertible object in the stable module
category of~$\StA(1)$. Recently it has
been shown that no analogues exist for
$\StA(n)$ with $n\geq2$. In previous work 
the author used doubling to produce an 
`iterated double Joker' which is an 
$\StA(n)$-module but not stably invertible. 
We also showed that for $n=1,2,3$ these 
iterated doubles were realisable as 
cohomology of CW spectra, but no such 
realisation existed for~$n>3$.

The main point of this paper is to show
that in the height~$2$ chromatic context,
the Morava $K$-theory of double Jokers
realises an exceptional endotrivial module
over the quaternion group of order~$8$
that only exists over a field of
characteristic~$2$ containing a primitive
cube root of unity. This has connections
with certain Massey products in the
cohomology of the quaternion group.
\end{abstract}

\maketitle

\section*{Introduction}

Following Adams \& Priddy~\cite{JFA&SBP},
in \cites{AB:Jokers,AB&TB:Jokers} we
considered the Joker $\StA(1)$-module
and its iterated doubles over the finite
subHopf algebras $\StA(n)\subseteq\StA$,
and showed that for small values of~$n$,
there were spectra and spaces realising
these. From an algebraic point of view,
the original~$\StA(1)$ Joker module was
important because it gave a non-trivial
self inverse stably invertible module, 
i.e., an element of order~$2$ in the 
Picard group of the stable module 
category of~$\StA(1)$. More recently,
Bhattacharya \& Ricka~\cite{PB&NR:PicA2}
and Pan \& Yan have shown that no such 
exotic elements can exist for~$\StA(n)$ 
when $n\geq2$ making use of ideas found 
in the related study of endotrivial 
modules for group algebras, conveniently 
described in the recent book of 
Mazza~\cite{NM:EndoTrivBook}.

The main aim of this paper is to show
that at least some of our geometric
Joker spectra have Lubin-Tate cohomology
which realises a certain lifting of
a $5$-dimensional endotrivial module
over the quaternion group~$Q_8$ and
the field~$\F_4$. Here~$Q_8$ is
realised as a subgroup of the second
Morava stabilizer group chromatic.
This example suggests that in the
chromatic setting there may be other
interesting endotrivial modules
associated with finite subgroups of
Morava stabilizer groups; Lennart
Meier has pointed out that this fits
well with results in
\cite{DC-AM-NN-JN:Descent}*{appendix~B}.

We collect some useful algebraic ideas 
and results on skew group rings and 
skew Hecke algebras in the Appendix.

\noindent
\textbf{Conventions and notation:}
We will work at the prime $p=2$ and
chromatic height~$2$ when considering
stable homotopy theory.


\section{Homotopy fixed points for
finite subgroups of Morava stabilizer
groups}\label{sec:HtpyFixPts}

We briefly recall the general set-up
for homotopy fixed point spectra of
Lubin-Tate spectra, where the group
involved is finite, although work of
Devinatz \& Hopkins~\cite{ESD-MJH:HtpyFixPtSpectraClSubgpsMoravaStabGp}
allows for more general subgroups of
Morava stabilizer groups to be used.
We will adopt the notation of
Henn~\cite{H-WH:CentResn}; in particular,
$\mathbb{G}_n$ is the \emph{extended
Morava stabilizer group}
\[
\mathbb{G}_n =
\mathbb{D}_n^\times/\langle S^n\rangle
\iso
\Gal(\F_{p^n}/\F_p)\ltimes\mathcal{O}_n^\times,
\]
where $S\in\mathbb{D}_n$ is the uniformizer
satisfying $S^n=p$.
\begin{examp}\label{examp:CentralC_2}
For any prime $p$ and $n\geq1$, there is
a unique central subgroup of order~$2$,
namely $C_2=\{\pm1\}\lhd\mathbb{G}_n$.
When $n=1$ and $p=2$, it is well known
that $E_1^{C_2}\sim K\mathrm{O}_2$.

For $p$ odd, there is a unique central
cyclic subgroup $C_{p-1}\lhd\mathbb{G}_n$
of order $p-1$, and when $n=1$ $E_1^{hC_{p-1}}$
is the Adams summand of $K\mathrm{U}_p$.
\end{examp}

\begin{examp}\label{examp:Q8}
When $p=2=n$, $\mathrm{O}_2^\times$ contains
a subgroup~$G_{24}$ of order~$24$ whose
unique $2$-Sylow subgroup is isomorphic
to the quaternion group~$Q_8$; this is 
the \emph{binary tetrahedral group} and 
double covers $A_4\leq\SO(3)$, the group 
of rotational symmetries of a regular 
tetrahedron. This group is the semidirect 
product $C_3\ltimes Q_8$ and there is also
a split extension
\[
G_{48}=\Gal(\F_{4}/\F_2)\ltimes G_{24}
\leq \mathbb{G}_2
\]
of order~$48$ in the extended Morava
stabilizer group. The fixed point spectrum
$E_2^{hG_{48}}$ is an avatar of the
spectrum of topological modular forms;
see the article by Hopkins \& Mahowald
in \cite{TMFbook}*{part~III}. A subgroup
$H\leq G_{48}$ gives rise to extensions
$E_2^{hG_{48}}\to E_2^{hH}\to E_2$ where
the latter is a faithful $H$-Galois
extension in the sense of Rognes~\cite{JR:MAMS192};
this depends on work of Devinatz \&
Hopkins~\cite{ESD-MJH:HtpyFixPtSpectraClSubgpsMoravaStabGp}.
\end{examp}

\section{A finite group of operations in Lubin-Tate
theory of height $2$}\label{sec:L-TOps}

Our work requires an explicit realisation of
$Q_8$ as a subgroup of the height~$2$ Morava
stabilizer group. We follow the account and
notation of Henn~\cite{H-WH:CentResn}*{section~2},
especially lemma~2.1.

The ring of \emph{Hurwitz quaternion}s $\mathcal{H}$
is the subdomain of $\mathbb{H}$ additively
generated by the elements
\[
\dfrac{(\pm1\pm i\pm j\pm k)}{2}.
\]
It has a unique completely prime maximal
ideal $\mathcal{M}$ which contains~$2$
as well as~$i+1,j+1,k+1$. The quotient
ring is a field with~$4$ elements,
\[
\F_4=\mathcal{H}/\mathcal{M}=\F_2(\omega),
\]
where $\omega$ denotes (the residue class of)
the primitive cube root of unity
\[
\omega = -\frac{(1+i+j+k)}{2}.
\]
Routine calculations show that
\[
i\omega i^{-1} = \omega + j+k \equiv \omega\bmod{\mathcal{M}}
\]
and also
\[
j\omega j^{-1} \equiv \omega \equiv k\omega k^{-1}\bmod{\mathcal{M}},
\]
therefore the quaternion subgroup
$Q_8=\langle i,j\rangle\leq\mathcal{H}^\times$
acts trivially of $\F_4$ and we may form the
(trivially twisted) group ring
$\F_4\langle Q_4\rangle=\F_4[Q_4]$.

We can complete $\mathcal{H}$ with respect
to $\mathcal{M}$ or equivalently~$2$, to
obtain a model for the maximal order
$\mathcal{O}_2$ in the division algebra
$\mathbb{D}_2=\mathcal{H}_{\mathcal{M}}$.
In fact
\[
\mathbb{D}_2 = \Z_4\langle S\rangle/(S^2-2),
\]
where $\Z_4=\mathrm{W}(\F_4)=\Z_2(\omega)$
is the ring of Witt vectors for $\F_4$ and
the uniformizer $S$ intertwines with $\Z_4$
so that $S(-)S^{-1}$ is the lift of Frobenius
(and so $S^2$ acts trivially). The quotient
group
\[
\mathbb{G}_2 =
\mathbb{D}_2^\times/\langle S^2\rangle
\iso
\Gal(\F_4/\F_2)\ltimes\mathcal{O}_2^\times
\]
is the \emph{extended Morava stabilizer group}.

Here is an explicit description for elements
of $Q_8$ in terms of Teichm\"uller expansions
as in~\cite{H-WH:CentResn}*{lemma~2.1}:
\begin{equation}\label{eq:Q8explicit}
i = \frac{1}{3}(1+2\omega^2)(1-aS),
\quad
j = \frac{1}{3}(1+2\omega^2)(1-a\omega^2S),
\quad
k = \frac{1}{3}(1+2\omega^2)(1-a\omega S),
\end{equation}
where we choose $\sqrt{-7}\in\Z_2$ to be
the square root of $-7$ satisfying
$\sqrt{-7}\equiv 5\bmod{8}$ and set
\[
a = \frac{1-2\omega}{\sqrt{-7}}\in\Z_4.
\]
%
Notice that working modulo $S^3=2S$ in
$\mathcal{O}_2$,
\begin{equation}\label{eq:Q8explicit-mod}
i \equiv 1+S+2\omega,
\quad
j \equiv 1+\omega^2S+2\omega,
\quad
k \equiv 1+\omega S+2\omega.
\end{equation}

Of course there is a twisted group ring
$(E_2)_0\langle Q_8\rangle$ which
has $\F_4[Q_4]$ as a quotient ring.

\section{Lubin-Tate theory for double Joker
spectra}\label{sec:L-TDoubleJokers}

Let $J=J(2)$ be one of the finite CW spectra
constructed in~\cite{AB:Jokers}. Its mod~$2$
cohomology is the cyclic $\StA(2)$-module
$H^*(J)\iso\StA(2)/\StA(2)\{\mathrm{Q}^0,\mathrm{Q}^1,\mathrm{Q}^2,\Sq^6\}$
(here the $\mathrm{Q}^i$ are the Milnor
primitives), and there are two possible
extensions to an $\StA$-module with trivial
or non-trivial $\Sq^8$-action giving dual
$\StA$-modules.
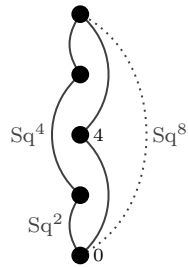
\begin{center}
\begin{tikzpicture}[scale=0.8]
\Vertex[y=0,size=.05,color=black]{B2}
\Vertex[y=-1,size=.05,color=black]{B1}
\Vertex[y=-2,size=.05,color=black]{B0}
\Text[y=-2,position=right,distance=1mm]{\tiny$4$}
\Vertex[y=-3,size=.05,color=black]{B-1}
\Vertex[y=-4,size=.05,color=black]{B-2}
\Text[y=-4,position=right,distance=1mm]{\tiny$0$}
\Edge[lw=0.75pt,bend=-45,position=right](B0)(B2)
\Edge[lw=0.75pt,bend=30](B1)(B2)
\Edge[lw=0.75pt,bend=45,label={$\Sq^4$},position=left](B-1)(B1)
\Edge[lw=0.75pt,bend=30,label={$\Sq^2$},position=left](B-2)(B-1)
\Edge[lw=0.75pt,bend=-45](B-2)(B0)
\Edge[lw=0.75pt,bend=-60,style=dotted,label={$\Sq^8$},position=right](B-2)(B2)
\end{tikzpicture}
\end{center}
The attaching maps in such a CW spectrum 
are essentially suspensions of $\eta$ and
$\nu$. Up to homotopy equivalence there
are two such spectra which are Spanier-Whitehead
dual to each other and realise the two
$\StA$-module extensions.

There is a CW spectrum $dA(1)$ known as `the
double of $\StA(1)$' whose cohomology as an
$\StA(2)$-module is $H^*(dA(1))\iso\StA(2)/\!/\StE(2)$;
for a detailed discussion see
Bhattacharya et al~\cite{BEM:v2-periodicA1}.
In \cite{AB:Jokers}*{remark~5.1} we outlined
how to construct such a spectrum starting with
a double Joker and attaching cells. By construction,
$dA(1)$ contains $J$ as a subcomplex with cofibre
a suspension of the `upside-down double question
mark' complex $Q^\text{\textquestiondown}$ whose
cohomology is $3$-dimensional and has a non-trivial
action of $\Sq^2\Sq^4$.
\begin{center}
\begin{tikzpicture}[scale=0.8]
\Text[x=-3,y=1.5]{$H^*(Q^\text{\textquestiondown})$}
\Vertex[y=3,size=.05,color=black]{B6}
\Vertex[y=2,size=.05,color=black]{B4}
\Vertex[y=0,size=.05,color=black]{B0}
\Edge[lw=0.75pt,bend=30,label={$\Sq^2$},position=left](B4)(B6)
\Edge[lw=0.75pt,bend=45,label={$\Sq^4$},position=left](B0)(B4)
\end{tikzpicture}
\end{center}
This is stably Spanier-Whitehead dual
to the`double question mark' complex
$Q^\text{?}$ whose cohomology has a
non-trivial action of~$\Sq^4\Sq^2$.
\begin{center}
\begin{tikzpicture}[scale=0.8] 
\Text[x=-3,y=1.5]{$H^*(Q^\text{?})$}
\Vertex[y=3,size=.05,color=black]{B6}
\Vertex[y=1,size=.05,color=black]{B2}
\Vertex[y=0,size=.05,color=black]{B0}
\Edge[lw=0.75pt,bend=-30,label={$\Sq^4$},position=right](B2)(B6)
\Edge[lw=0.75pt,bend=-45,label={$\Sq^2$},position=right](B0)(B2)
\end{tikzpicture}
\end{center}

There is cofibre sequence
$J\to dA(1)\to \Sigma^6Q^\text{\textquestiondown}$
of the form shown.
\begin{center}
\begin{tikzpicture}[scale=0.8]
\Vertex[x=-4,y=4,size=.1,color=white]{J4}
\Vertex[x=-4,y=3,size=.1,color=white]{J32}
\Vertex[x=-4,y=2,size=.1,color=white]{J2}
\Vertex[x=-4,y=1,size=.1,color=white]{J1}
\Vertex[x=-4,y=0,size=.1,color=white]{J0}
\Edge[lw=0.75pt,bend=-30](J32)(J4)
\Edge[lw=0.75pt,bend=45](J2)(J4)
\Edge[lw=0.75pt,bend=-45](J1)(J32)
\Edge[lw=0.75pt,bend=-30](J0)(J1)
\Edge[lw=0.75pt,bend=45](J0)(J2)

\Vertex[x=-3.0,y=2,Pseudo]{P0}
\Vertex[x=-1.0,y=2,Pseudo]{P1}
\Edge[Direct](P0)(P1)

\Vertex[x=2,y=6,color=white,size=.05]{A6}
\Text[x=2,y=6,position=right,distance=1mm]{\tiny$12$}
\Vertex[x=2,y=5,color=white,size=.1]{A5}
\Vertex[x=2,y=4,size=.1,color=white]{A4}
\Vertex[y=3,size=.1,color=white]{A31}
\Vertex[x=2,y=3,size=.1,color=white]{A32}
\Vertex[y=2,size=.1,color=white]{A2}
\Vertex[y=1,size=.1,color=white]{A1}
\Vertex[y=0,size=.1,color=white]{A0}
\Text[y=0,position=right,distance=1mm]{\tiny$0$}
\Edge[lw=0.75pt,bend=30](A5)(A6)
\Edge[lw=0.75pt,bend=-45](A4)(A6)
\Edge[lw=0.75pt,bend=-30](A32)(A4)
\Edge[lw=0.75pt](A31)(A5)
\Edge[lw=0.75pt,bend=30](A2)(A31)
\Edge[lw=0.75pt](A2)(A4)
\Edge[lw=0.75pt](A1)(A32)
\Edge[lw=0.75pt,bend=-30](A0)(A1)
\Edge[lw=0.75pt,bend=45](A0)(A2)

\Vertex[x=3.0,y=4,Pseudo]{PQ0}
\Vertex[x=5.0,y=4,Pseudo]{PQ1}
\Edge[Direct](PQ0)(PQ1)

\Vertex[x=6,y=6,color=white,size=.05]{Q6}
\Vertex[x=6,y=5,color=white,size=.1]{Q5}
\Vertex[x=6,y=3,size=.1,color=white]{Q31}
\Text[x=6,y=3,position=right,distance=1mm]{\tiny$6$}
\Edge[lw=0.75pt,bend=-45](Q31)(Q5)
\Edge[lw=0.75pt,bend=-30](Q5)(Q6)
\end{tikzpicture}
\end{center}

We can apply a complex oriented homology
theory to this cofibre sequence, thus
obtaining a short exact sequence; in
particular we will apply $BP_*(-)$,
$(E_2)_*(-)$ or $(K_2)_*(-)$. Our goal
is to understand the Lubin-Tate cohomology
$E_2^*(J)$ as a left $E_2^*\langle Q_8\rangle$-module
where $Q_2\leq\mathbb{G}_2$ is a quaternion
subgroup. Since~$E_2^*(J)$ is a finitely
generated free module and~$J$ is dualizable,
we can instead work with right module
$(E_2)_*(J)$ in terms of the corresponding
$(E_2)_*(E_2)$-comodule structure. Actually
we prefer to work directly with the smaller
complex $Q^\text{\textquestiondown}$ and
use the fact $(E_2)_*(dA(1))$ and $(E_2)^*(dA(1))$
are free $E_2^*\langle Q_8\rangle$-modules
of rank~$1$: this is well-known and appears
in section~6 of the article by Hopkins 
\& Mahowald~\cite{TMFbook}*{part~III},
and a detailed discussion also occurs in
\cite{RRB&JR:tmfBook}*{section~1.4}. The
key point is to use the equivalence
$E_2^{hQ_8}\wedge dA(1)\sim E_2$ together 
with results of Devinatz \&
Hopkins~\cite{ESD-MJH:HtpyFixPtSpectraClSubgpsMoravaStabGp}.
So our main calculational result identifies
the right $E_2^*\langle Q_8\rangle$-module
$(E_2)_*(Q^\text{\textquestiondown})$; we
will do this by first describing the
$K_2^*[Q_8]$-module
$(K_2)_*(Q^\text{\textquestiondown})$.

Here is our main result.
\begin{thm}\label{thm:JokerInvtble}
The $E_2^*\langle Q_8\rangle$-module $E_2^*(J)$
is stably invertible and self dual, and its
reduction to~$K_2^*(J)$ is a $5$-dimensional
stably invertible $K_2^*[Q_8]$-module.
\end{thm}

Of course we can reduce to studying $K_2^0(J)$
as a $K_2^0[Q_8]=\F_4[Q_8]$-module. The
$5$-dimensional stably invertible $\F_4[Q_8]$-module
$W_5$ is that of\/~\cite{NM:EndoTrivBook}*{theorem~3.8(1)}
and this is $\Omega W_3$ for a $3$-dimensional 
stably invertible $\F_4[Q_8]$-module~$W_3$ 
which we will show is isomorphic to 
$(K_2)_0(Q^\text{\textquestiondown})$. For 
a suitable choice of basis $w_1,w_2,w_3$, 
the action of $Q_8$ on~$W_3$ is given by
\begin{equation}\label{eq:W3-action}
\left\{
\begin{aligned}
\quad 
iw_1 &= w_1 + w_2, &\quad jw_1 &= w_1 + \omega w_2, 
\\
\quad 
iw_2 &= w_2 + w_3, &\quad jw_2 &= w_2 + \omega^2 w_2, 
\\
\quad 
iw_3 &= w_3, &\quad jw_1 &= w_3, \\
\end{aligned}
\right.
\end{equation}
with corresponding matrices
\[
i\: \begin{bmatrix} 1 & 0 & 0 \\ 1 & 1 & 0 \\ 0 & 1 & 1 \end{bmatrix},
\quad
j\: \begin{bmatrix} 1 & 0 & 0 \\ \omega & 1 & 0 \\ 0 & \omega^2 & 1 \end{bmatrix}.
\]
In fact a different choice of basis 
\begin{equation}\label{eq:W3-newbasis}
w'_1 = w_1+w_2,\quad w'_2=w_2+\omega^2w_3,\quad w'_3=w_3
\end{equation}
is more convenient for our purposes; the 
corresponding matrices are then
\begin{equation}\label{eq:W3-newmatrices}
i\: \begin{bmatrix} 1 & 0 & 0 \\ 1 & 1 & 0 \\ \omega & 1 & 1 \end{bmatrix},
\quad
j\: \begin{bmatrix} 1 & 0 & 0 \\ \omega & 1 & 0 \\ \omega & \omega^2 & 1 \end{bmatrix}.
\end{equation}

The lifting of results to the Lubin-Tate
setting uses the algebra discussed in
Appendix~\ref{app:CrossProd}. Of course 
we need to do some topological calculations
to obtain these results and these are 
outlined in the next section.


%
%

\section{Calculations}\label{sec:Calculations}
\subsection*{Homological algebra conventions}
Before describing the calculations required,
we explain our notational conventions for
homological algebra.

Given a flat Hopf algebroid $(A,\Gamma)$ which
might be graded and two left $\Gamma$-comodules
$L,M$, we denote by $\Cohom_\Gamma(L,M)$ the
set of comodule homomorphisms $L\to M$, and
$\Coext^{s,*}_\Gamma(L,-)$ for the $s$-th
right derived functor of $\Cohom^*_\Gamma(L,-)$
where $*$ indicates the internal degree shift.
When the grading is trivial (i.e., concentrated
in degree $0$) we write $\Cohom^{s}_\Gamma(L,-)$
and $\Coext^{s}_\Gamma(L,-)$.

If $G$ is a finite group and $R$ is a (possibly
graded) commutative ring, then the group ring
$R[G]$ is a Hopf algebra over $R$ and its dual
$R(G)=\Hom_R(R[G],R)$ forms a commutative Hopf
algebra $(R(G),R)$. Every left $R(G)$-comodule
becomes a right $R[G]$-module in a natural way, 
and vice versa. Moreover, for a left $R(G)$-comodule
$L$ there is a natural isomorphism
\[
\Cohom^*_{R(G)}(L,-) \iso \Hom^*_{R[G]}(L,-)
\]
and this induces natural isomorphisms of right
derived functors
\begin{equation}\label{eq:Coext-Ext}
\Coext^{s,*}_{R(G)}(L,-) \iso \Ext^{s,*}_{R[G]}(L,-).
\end{equation}
Finally, if $R$ is a graded ring of the form
$R=\k[w,w^{-1}]$ where $w$ has even positive
degree, then $R[G]=R\otimes_{\k}\k[G]$ and
$R(G)=R\otimes_{\k}\k(G)$. When $L$ has the
form $L=R\otimes_\k L_0$, then
\begin{equation}\label{eq:Ext-ungraded}
\Ext^{s,*}_{R[G]}(L,-) \iso
\Ext^{s,*}_{\k[G]}(L_0,-).
\end{equation}

\subsection*{Comodules for some iterated
mapping cones}
We begin by recalling that the mapping cones
of the Hopf invariant $1$ elements have the
following $BP$-homology as $BP_*(BP)$-comodules,
where $x_k$ has degree $k$ and $x_0$ is coaction
primitive. Here
\[
BP_*(\cone(\eta)) = BP_*\{x_0,x_2\},
\quad
BP_*(\cone(\nu)) = BP_*\{x_0,x_4\},
\quad
BP_*(\cone(\sigma)) = BP_*\{x_0,x_8\},
\]
with
\begin{subequations}\label{eq:BPcomodules}
\begin{align}
\psi(x_2) &= t_1\otimes x_0 + 1\otimes x_2,
\label{eq:BPcomodules-eta} \\
\psi(x_4) &= (v_1t_1+t_1^2)\otimes x_0 + 1\otimes x_4.
\label{eq:BPcomodules-nu} \\
\psi(x_8) &=
(v_{2} t_{1}- 3 t_{1}^{4}
-v_{1}^{3} t_{1} - 4 v_{1}^{2} t_{1}^{2}
- 5 v_{1} t_{1}^{3}
+ v_{1} t_{2} + 2 t_{1} t_{2})\otimes x_0 + 1\otimes x_8
\label{eq:BPcomodules-sigma} \\
&\equiv (v_{2} t_{1}+ t_{1}^{4})\otimes x_0 + 1\otimes x_8 \mod{(2,v_1)}. \notag
\end{align}
\end{subequations}
Such formulae are well-known and follow from
the fact that these homotopy elements are
detected by elements that originate in the
chromatic spectra sequence on
\begin{multline*}
v_1/2\in\Coext_{BP_*(BP)}^{0,2}(BP_*,BP_*/2^\infty),
\quad
v_1^2/4\in\Coext_{BP_*(BP)}^{0,4}(BP_*,BP_*/2^\infty), \\
(v_1^4+8v_1v_2)/16\in\Coext_{BP_*(BP)}^{0,8}(BP_*,BP_*/2^\infty);
\ph{\Coext_{BP_*(BP)}^{0,8}(BP_*,BP_*/2^\infty}
\end{multline*}
see \cites{MRW:PerPhenANSS,DCR:NovicesGuide}
for details.

We require a computational result that ought 
to be standard but we do not know a convenient 
reference.
\begin{lem}\label{lem:IteratedMappingCones}
For $BP_*(S^0\cup_\nu e^4\cup_\eta e^6)$ 
there is a $BP_*$-basis $x_0,x_4,x_6$ with 
$BP_*(BP)$-coaction given by
\begin{align*}
\psi(x_0) &= 1\otimes x_0, \\
\psi(x_4) &= (v_1t_1+t_1^2)\otimes x_0 + 1\otimes x_4, \\
\psi(x_6) &= 
\bigl(t_2+(2/3)t_1^3+v_1t_1^2\bigr)\otimes x_0 + t_1\otimes x_4 + 1\otimes x_6.
\end{align*}
\end{lem}
\begin{proof}
We only need to verify the last coaction
and only the term involving $x_0$ is 
unclear. Suppose that
\[
\psi(x_6) = \theta\otimes x_0 + t_1\otimes x_4 + 1\otimes x_6.
\]
Then by coassociativity we obtain
\begin{multline*}
\psi(\theta)\otimes x_0 
+ 1\otimes t_1\otimes x_4 + t_1\otimes 1\otimes x_4
+ 1\otimes1\otimes x_6              \\
=
\theta\otimes1\otimes x_0 
+ t_1\otimes(v_1t_1+t_1^2)\otimes x_4 + t_1\otimes1\otimes x_4
+ 1\otimes\theta\otimes x_0 + 1\otimes t_1\otimes x_4 + 1\otimes1\otimes x_6
\end{multline*}
and so
\begin{align*}
\psi(\theta) &= 
1\otimes\theta + t_1\otimes(v_1t_1+t_1^2) + \theta\otimes1  \\
&= 
1\otimes\theta + t_1(v_1+2t_1)\otimes v_1t_1 + t_1\otimes t_1^2 + \theta\otimes1  \\
&= 
1\otimes\theta + v_1t_1\otimes t_1 + 2t_1^2\otimes t_1 + t_1\otimes t_1^2 + \theta\otimes1.
\end{align*}
A calculation shows that 
\[
\psi(t_2+(2/3)t_1^3+v_1t_1^2) =
1\otimes(t_2+(2/3)t_1^3+v_1t_1^2) 
+ t_1\otimes(v_1t_1+t_1^2) 
+(t_2+(2/3)t_1^3+v_1t_1^2)\otimes1.
\]
So we obtain the formulae stated.
\end{proof}

Now given a map of ring spectra $BP\to E$, where
$E$ is Landweber exact, these $BP_*(BP)$-comodules
map to $E_*(E)$-comodules $E_*(\cone(\eta))$ and
$E_*(\cone(\nu))$. Our main interest will focus
on the examples $E=E_2$ (the height $2$ Lubin-Tate 
spectrum) and $E=K_2$ (the $2$-periodic Morava 
$K$-theory spectrum with coefficients in $\F_4$). 
In the latter case we have $(K_2)_* = \F_4[u,u^{-1}]$ 
where $u\in(K_2)_2$ (so $u^3=v_2$) and we set
$\bar{x}_{2k}=u^{-k}x_{2k}\in(K_2)_0(\cone(\gamma))$
when $\gamma=\eta,\nu$. The Hopf algebroid
here is
\[
(K_2)_*(E_2) =
(K_2)_*[\alpha_r : r\geq0]/(\alpha_0^{3}-1,\,\alpha_r^4-\alpha_r : r\geq1),
\]
where the right unit on $u$ is $\eta_{\mathrm{r}}(u)=u\alpha_0$
and the image of $t_k\in BP_{2^{k+1}-2}(BP)$
is
\[
u^{2^{k}-1}\alpha_k\in(K_2)_{2^{k+1}-2}(E_2).
\]
It is standard that every element of $\mathcal{O}_2^\times$
has a unique series expansion as $\sum_{r\geq0}a_rS^r$,
where the Teichm\"uller representatives~$a_r$ 
satisfy
\[
a_0^{3}=1,\qquad a_r^4=a_r\quad(r\geq1).
\]
Then we may identify $(K_2)_0(E_2)$ with
the algebra of continuous maps $\mathcal{O}_2^\times\to\F_4$
and then $\alpha_k$ is identified with the
locally constant function given by
\[
\alpha_k\biggl(\sum_{r\geq0}a_rS^r\biggr) = a_k.
\]

The left $(K_2)_*(E_2)$-coaction on a
comodule $M_*$ induces an adjoint right
action of $\mathcal{O}_2^\times$. For
any finite subgroup $G\leq\mathcal{O}_2^\times$
there is an induced action of the skew
group ring $(K_2)^*\langle G\rangle$.
This also gives a right action of
$\F_4\langle G\rangle$ on each $M_k$.
Of course we are using the $(K_2)_*$-linear
pairing $M_*\otimes_{(K_2)_*}M^*\to(K_2)_*$
to define this. Standard linear algebra
says that when $M_*$ is finite dimensional
over $(K_2)_*$, given a basis for $M_*$ and
the dual basis for $M^*=\Hom_{(K_2)_*}(M_*,(K_2)_*)$,
the matrices for expressing the action
on $M_*$ and its adjoint action on~$M^*$
are mutually transpose.

\subsection*{$(K_2)_0(\cone(\eta))$}
Here we have the coaction formulae
\[
\bar{x}_0\mapsto 1\otimes\bar{x}_0,
\quad
\bar{x}_2\mapsto \alpha_1\otimes\bar{x}_0 + \alpha_0\otimes\bar{x}_2,
\]
and the right action $Q_8$  has matrix
representations with respect to the basis
$\bar{x}_0,\bar{x}_2$ obtained from the
above discussion together
with~\eqref{eq:Q8explicit}
and~\eqref{eq:Q8explicit-mod}.
\[
i\: \begin{bmatrix} 1 & 1 \\ 0 & 1 \end{bmatrix},
\quad
j\: \begin{bmatrix} 1 & \omega^2 \\ 0 & 1 \end{bmatrix}.
\]

\subsection*{$(K_2)_0(\cone(\nu))$}
The coaction is
\[
\bar{y}_0\mapsto 1\otimes\bar{y}_0,
\quad
\bar{y}_4\mapsto \alpha_1^2\otimes\bar{y}_0 + \alpha_0^2\otimes\bar{y}_4,
\]
and the matrix representation is
\[
i\: \begin{bmatrix} 1 & 1 \\ 0 & 1 \end{bmatrix},
\quad
j\: \begin{bmatrix} 1 & \omega \\ 0 & 1 \end{bmatrix}.
\]

\subsection*{$(K_2)_0(S^0\cup_\nu e^4\cup_\eta e^6)$}
Using Lemma~\ref{lem:IteratedMappingCones}, we can
find a basis
$\bar{z}_0,\bar{z}_4,\bar{z}_6\in(K_2)_0(S^0\cup_\nu e^4\cup_\eta e^6)$
and
\begin{equation}\label{eq:DbleUpside?}
\bar{z}_0\mapsto 1\otimes\bar{z}_0,
\quad
\bar{z}_4\mapsto \alpha_1^2\otimes\bar{z}_0 + \alpha_0^2\otimes\bar{z}_4,
\quad
\bar{z}_6\mapsto \alpha_2\otimes\bar{z}_0 + \alpha_0^2\alpha_1\otimes\bar{z}_4 + 1\otimes\bar{z}_6,
\end{equation}

\[
i\: \begin{bmatrix} 1 & 1 & \omega \\ 0 & 1 & 1 \\ 0 & 0 & 1 \end{bmatrix},
\quad
j\: \begin{bmatrix} 1 & \omega & \omega \\ 0 & 1 & \omega^2 \\ 0 & 0 & 1 \end{bmatrix}.
\]
These are the matrices for the adjoint 
of the representation~$W_3$ in terms of 
the basis in~\eqref{eq:W3-newbasis},
i.e., the transposes of the matrices 
in~\eqref{eq:W3-newmatrices}.

\subsection*{$K_2^0(\cone(\sigma))$}
Here the relation $\alpha_1^{4} + \alpha_1=0$
gives
\[
\bar{x}_0\mapsto 1\otimes\bar{x}_0,
\quad
\bar{x}_8\mapsto
(\alpha_1^{4} + \alpha_1)\otimes\bar{x}_0 + \alpha_0\otimes\bar{x}_8
= \alpha_0\otimes\bar{x}_8,
\]
so $i,j$ act trivially.

\subsection*{$(K_2)_0(S^0\cup_\sigma e^8\cup_\nu e^{12})$}
Here we have a basis
$\bar{z}_0,\bar{z}_8,\bar{z}_{12}\in(K_2)_0(S^0\cup_\sigma e^8\cup_\nu e^{12})$
and
\[
\bar{z}_0\mapsto 1\otimes\bar{z}_0,
\quad
\bar{z}_8\mapsto \alpha_0\otimes\bar{z}_8,
\quad
\bar{z}_{12}\mapsto *\otimes\bar{z}_0 + \alpha_1^2\otimes\bar{z}_8 + \alpha_0^2\otimes\bar{z}_{12},
\]
\[
i\: \begin{bmatrix} 1 & 0 & * \\ 0 & 1 & 1 \\ 0 & 0 & 1 \end{bmatrix},
\quad
j\: \begin{bmatrix} 1 & 0 & * \\ 0 & 1 & \omega \\ 0 & 0 & 1 \end{bmatrix}.
\]
Here $Z(Q_8)$ acts trivially so the
representation factors through the
abelianisation, hence this does not 
give a stably invertible $Q_8$-module.
The precise form of the starred terms 
can be determined by a similar method 
to that used in the proof of 
Lemma~\ref{lem:IteratedMappingCones}.
\bigskip
%
%
%

\section{More on modular representations 
of $Q_8$}
\label{sec:Q8reps} 

The reader may find it useful to relate 
the results in this section to 
Ravenel~\cite{DCR:CohomMoravaStabAlgs}*{proposition~3.5}.

Recall that for any field $\mathbf{k}$ of
characteristic~$2$, the cohomology of
$\mathbf{k}[Q_8]$ has the form
\begin{equation}\label{eq:ExtQ8}
\Ext_{\mathbf{k}[Q_8]}^*(\mathbf{k},\mathbf{k})
=
\mathbf{k}[\mathrm{u},\mathrm{v},\mathrm{w}]/
(\mathrm{u}^2+\mathrm{u}\mathrm{v}+\mathrm{v}^2,
\mathrm{u}^2\mathrm{v}+\mathrm{u}\mathrm{v}^2,
\mathrm{u}^3,\mathrm{v}^3),
\end{equation}
where $\mathrm{u},\mathrm{v}$ have degree~$1$
and $\mathrm{w}$ has degree~$4$; see for example
Adem \& Milgram~\cite{AA-RJM:CohomFinGps}*{lemma~IV.2.10}.

Of course
$\Ext_{\mathbf{k}[Q_8]}^1(\mathbf{k},\mathbf{k})$
can be identified with the group of all
homomorphisms $Q_8\to\mathbf{k}$ into
the additive group of~$\mathbf{k}$. We
will need to make explicit choices for
the generators and we define them to
be the homomorphisms
$\mathrm{u},\mathrm{v}\:Q_8\to\mathbf{k}$
given by
\[
\mathrm{u}(i)=1,\;\mathrm{u}(j)=0,
\;\mathrm{v}(i)=0,\;\mathrm{v}(j)=1.
\]
The functions $\alpha_1,\alpha_1^2\:Q_8\to\mathbf{k}$
are also homomorphisms and can be expressed
as
\[
\alpha_1 = \mathrm{u}+\omega^2\mathrm{v},\;
\alpha_1^2 = \mathrm{u}+\omega\mathrm{v}.
\]

Notice that when $\mathbf{k}$ does not
contain a primitive cube root of unity,
$\mathrm{u}^2+\mathrm{u}\mathrm{v}+\mathrm{v}^2$
does not factor, but if $\omega\in\mathbf{k}$
is a primitive cube root of unity then
\[
\mathrm{u}^2+\mathrm{u}\mathrm{v}+\mathrm{v}^2
=
(\mathrm{u}+\omega\mathrm{v})(\mathrm{u}+\omega^2\mathrm{v}).
\]
This means that the Massey product
$\langle \mathrm{u}+\omega^2 \mathrm{v},
\mathrm{u}+\omega \mathrm{v},\mathrm{u}+\omega^2\mathrm{v}\rangle
\subseteq\Ext_{\mathbf{k}[Q_8]}^2(\mathbf{k},\mathbf{k})$
is defined and this has indeterminacy
$\mathbf{k}\{\mathrm{u}^2+\omega \mathrm{v}^2\}$.

The Massey product
\[
\langle[t_1],[v_1t_1+t_1^2],[t_1]\rangle
= \{[v_1t_1+t_1^2]^2\}
\subseteq\Coext_{BP_*(BP)}^{2,8}(BP_*,BP_*)
\]
corresponds to the Toda bracket
\[
\langle\eta,\nu,\eta\rangle=\{\nu^2\}\subseteq\pi_6(S).
\]
We can exploit naturality in cohomology of
Hopf algebroids together with \eqref{eq:Coext-Ext}
and \eqref{eq:Ext-ungraded} to obtain an algebra
homomorphism
\[
\Coext_{BP_*(BP)}^*(BP_*,BP_*)
\to\Ext_{(K_2)_*[Q_8]}^*((K_2)_*,(K_2)_*)
\xrightarrow{\iso}
(K_2)_*\otimes_{\F_4}\Ext_{\F_4[Q_8]}^*(\F_4,\F_4).
\]
Our calculations show that under this
\[
[t_1]\mapsto u(\mathrm{u}+\omega^2 \mathrm{v}),
\quad
[v_1t_1+t_1^2]\mapsto u^2(\mathrm{u}+\omega \mathrm{v}),
\]
hence
$\langle\mathrm{u}+\omega^2 \mathrm{v},
\mathrm{u}+\omega \mathrm{v},
\mathrm{u}+\omega^2 \mathrm{v}\rangle$
must contain $(\mathrm{u}+\omega\mathrm{v})^2
=\mathrm{u}^2+\omega^2\mathrm{v}^2$. It
follows that for any extension field
$\mathbf{k}$ of~$\F_4$,
\[
\langle \mathrm{u}+\omega^2 \mathrm{v},
\mathrm{u}+\omega \mathrm{v},
\mathrm{u}+\omega^2\mathrm{v}\rangle
=
\mathbf{k}\{\mathrm{u}+\omega\mathrm{v}\}+(\mathrm{u}^2+\omega^2\mathrm{v}^2)
\varsubsetneq
\Ext_{\mathbf{k}[Q_8]}^2(\mathbf{k},\mathbf{k}).
\]
Of course this could also be verified
directly using a good choice of
resolution of $\mathbf{k}$ over
$\mathbf{k}[Q_8]$.

We remark that the Massey product
\[
\langle[v_1t_1+t_1^2],[t_1],[v_1t_1+t_1^2]\rangle
\subseteq\Coext_{BP_*(BP)}^{2,10}(BP_*,BP_*)
\]
corresponds to the Toda bracket
\[
\langle\nu,\eta,\nu\rangle=\{\eta\sigma+\epsilon\}\subseteq\pi_8(S),
\]
and is related to the Massey product
\[
\langle \mathrm{u}+\omega\mathrm{v},
\mathrm{u}+\omega^2\mathrm{v},
\mathrm{u}+\omega\mathrm{v}\rangle
=
\mathbf{k}\{\mathrm{u}+\omega^2\mathrm{v}\}+(\mathrm{u}^2+\omega\mathrm{v}^2)
\varsubsetneq
\Ext_{\mathbf{k}[Q_8]}^2(\mathbf{k},\mathbf{k}).
\]

\subsection*{$5$-dimensional endotrivial modules 
for $\mathbf{k}[Q_8]$}
There are in fact two distinct $5$-dimensional 
endotrivial modules for $\mathbf{k}[Q_8]$ (this 
was pointed out to the author by Dave Benson) 
and we discuss some implications of this. We 
follow the notation of Dade~\cite{ECD:ExtnThmHalHigman}*{section~1}
with minor changes. 

In $\mathbf{k}[Q_8]$ we take the elements
\[
X= \omega i + \omega^2 j + k,
\quad
Y = \omega^2 i + \omega j + k
\]
which are in the augmentation ideal and 
satisfy the relations
\begin{equation}\label{eq:kQ8relations}
X^2 = YXY,
\quad 
Y^2 = XYX,
\quad 
XYXY = YXYX = \sum_{g\in Q_8}g,
\end{equation}
where the last element is a generator 
of the socle and so is an integral of 
the Hopf algebra $\mathbf{k}[Q_8]$.
This gives a $\mathbf{k}$-basis
\[
1,X,Y,YX,XY,XYX,YXY,XYXY=YXYX.
\]

The module $W_3$ has a basis $w_1,w_2,w_3$
for which the action of $Q_8$ is given by
\eqref{eq:W3-action}, so the actions of~$X$ 
and~$Y$ are given by
\[
\left\{
\begin{aligned}
\quad
Xw_1 &=0, &\quad Yw_1 &= \omega^2 w_2, \\
\quad
Xw_2 &= \omega w_3, &\quad Yw_2 &= 0, \\
\quad 
Xw_3 &= 0, &\quad Yw_3 &= 0. 
\end{aligned}
\right.
\]
This module is isomorphic to the cyclic quotient 
module
\[
\mathbf{k}[Q_8]/\mathbf{k}\{X,YX,XYX,YXY,XYXY\}.
\]
There is also the  cyclic quotient module
\[
\mathbf{k}[Q_8]/\mathbf{k}\{Y,XY,XYX,YXY,XYXY\}.
\]
These have the module structures shown where 
solid lines indicate multiplication by~$X$, 
dotted lines indicate multiplication by~$Y$ 
and the symbols indicate representatives of 
residue classes.
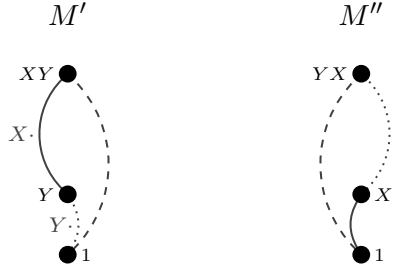
\begin{center}
\begin{tikzpicture}[scale=0.8]
\Text[y=0]{$M'$}
\Vertex[y=-1,size=.05,color=black]{B1}
\Text[y=-1,position=left,distance=1mm]{\tiny$XY$}
\Vertex[y=-3,size=.05,color=black]{B-1}
\Text[y=-3,position=left,distance=1mm]{\tiny$Y$}
\Vertex[y=-4,size=.05,color=black]{B-2}
\Text[y=-4,position=right,distance=1mm]{\tiny$1$}
\Edge[lw=0.75pt,bend=45,label={$X\cdot$},position=left](B-1)(B1)
\Edge[lw=0.75pt,bend=-30,label={$Y\cdot$},position=left,style=dotted](B-2)(B-1)
\Edge[lw=0.75pt,bend=-45,style=dashed](B-2)(B1)
\end{tikzpicture}
\qquad\qquad\qquad
\begin{tikzpicture}[scale=0.8]
\Text[y=0]{$M''$}
\Vertex[y=-1,size=.05,color=black]{B1}
\Text[y=-1,position=left,distance=1mm]{\tiny$YX$}
\Vertex[y=-3,size=.05,color=black]{B-1}
\Text[y=-3,position=right,distance=1mm]{\tiny$X$}
\Vertex[y=-4,size=.05,color=black]{B-2}
\Text[y=-4,position=right,distance=1mm]{\tiny$1$}
\Edge[lw=0.75pt,bend=-45,style=dotted](B-1)(B1)
\Edge[lw=0.75pt,bend=30](B-2)(B-1)
\Edge[lw=0.75pt,bend=45,style=dashed](B-2)(B1)
\end{tikzpicture}
\end{center}
In each case the central subalgebra $\mathbf{k}[Z(Q_8)]$
acts so that multiplication by~$i^2-1$ is
given by the dashed line. These are both 
endotrivial $\mathbf{k}[Q_8]$-modules by 
Chouinard's Theorem~\cite{NM:EndoTrivBook}*{theorem~2.1}. 
The $5$-dimensional modules~$\Omega M'$ 
and~$\Omega M''$ are also endotrivial.

There are two $3$-dimensional left ideals 
of $\mathbf{k}[Q_8]$,
\[
L' = \mathbf{k}[Q_8]\{XY\} = \mathbf{k}\{XY,YXY,XYXY\},
\quad
L'' = \mathbf{k}[Q_8]\{YX\} = \mathbf{k}\{YX,XYX,YXYX\},
\]
with  endotrivial quotient modules~$J'=\mathbf{k}[Q_8]/L'$
and~$J''=\mathbf{k}[Q_8]/L''$. Notice that 
$L'\iso M'$ and $L''\iso M''$, while~$J'$ 
and~$J''$ are both stably self-inverse.

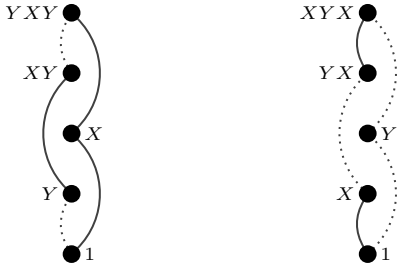
\begin{center}
\begin{tikzpicture}[scale=0.8]
\Vertex[y=0,size=.05,color=black]{B2}
\Text[y=0,position=left,distance=1mm]{\tiny$YXY$}
\Vertex[y=-1,size=.05,color=black]{B1}
\Text[y=-1,position=left,distance=1mm]{\tiny$XY$}
\Vertex[y=-2,size=.05,color=black]{B0}
\Text[y=-2,position=right,distance=1mm]{\tiny$X$}
\Vertex[y=-3,size=.05,color=black]{B-1}
\Text[y=-3,position=left,distance=1mm]{\tiny$Y$}
\Vertex[y=-4,size=.05,color=black]{B-2}
\Text[y=-4,position=right,distance=1mm]{\tiny$1$}
\Edge[lw=0.75pt,bend=-45,position=right](B0)(B2)
\Edge[lw=0.75pt,bend=30,style=dotted](B1)(B2)
\Edge[lw=0.75pt,bend=45](B-1)(B1) 
\Edge[lw=0.75pt,bend=30,style=dotted](B-2)(B-1) 
\Edge[lw=0.75pt,bend=-45](B-2)(B0)
\end{tikzpicture}
\qquad\qquad\qquad
\begin{tikzpicture}[scale=0.8]
\Vertex[y=0,size=.05,color=black]{B2}
\Text[y=0,position=left,distance=1mm]{\tiny$XYX$}
\Vertex[y=-1,size=.05,color=black]{B1}
\Text[y=-1,position=left,distance=1mm]{\tiny$YX$}
\Vertex[y=-2,size=.05,color=black]{B0}
\Text[y=-2,position=right,distance=1mm]{\tiny$Y$}
\Vertex[y=-3,size=.05,color=black]{B-1}
\Text[y=-3,position=left,distance=1mm]{\tiny$X$}
\Vertex[y=-4,size=.05,color=black]{B-2}
\Text[y=-4,position=right,distance=1mm]{\tiny$1$}
\Edge[lw=0.75pt,bend=-45,position=right,style=dotted](B0)(B2)
\Edge[lw=0.75pt,bend=30](B1)(B2)
\Edge[lw=0.75pt,bend=45,style=dotted](B-1)(B1)
\Edge[lw=0.75pt,bend=30](B-2)(B-1)
\Edge[lw=0.75pt,bend=-45,style=dotted](B-2)(B0)
\end{tikzpicture}
\end{center}

Clearly $J'$ and $J''$ are not isomorphic,
and from the known structure of the Picard 
group of the stable module category 
$\Pic(\mathbf{k}[Q_8])\iso C_4\times C_2$
we must have $J''\iso\Omega^2 J'$.

The module $J'$ corresponds to our double 
Joker complex, but~$J''$ seems not to be 
realisable as~$K_2^*(Z)$ for a CW spectrum.
The corresponding $\StA(2)$-module is 
$\Omega H^*(Q^\text{\textquestiondown})$
and there is no $\tmf$-module spectrum
$M$ for which 
$H_{\tmf}^*(M)\iso\Omega H^*(Q^\text{\textquestiondown})$,
in particular there is no CW spectrum $Z$ 
for which
\[
H_{\tmf}^*(\tmf\wedge Z)\iso H^*(Z)
\iso\Omega H^*(Q^\text{\textquestiondown})
\] 
as $\StA(2)$-modules.

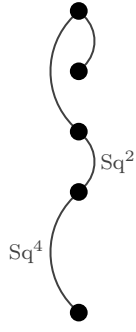
\begin{center}
\begin{tikzpicture}[scale=0.8]
\Text[y=6]{$\Omega H^*(Q^\text{\textquestiondown})$}
\Vertex[x=0,y=5,size=.1,color=black]{A10}
\Vertex[x=0,y=4,size=.1,color=black]{A8}
\Vertex[x=0,y=3,size=.1,color=black]{A6}
\Vertex[x=0,y=2,size=.1,color=black]{A4}
\Vertex[x=0,y=0,size=.1,color=black]{A0}

\Edge[lw=0.75pt,bend=-45](A8)(A10)
\Edge[lw=0.75pt,bend=45](A6)(A10)
\Edge[lw=0.75pt,bend=-45,label={$\Sq^2$},position=right](A4)(A6)
\Edge[lw=0.75pt,bend=45,label={$\Sq^4$},position=left](A0)(A4)
\end{tikzpicture}
\end{center}

\section{The action of $G_{24}$}\label{sec:G24action}
In this section we briefly discuss the 
action of the group $G_{24}$ of order~$24$
discussed in Example~\ref{examp:Q8}. This 
a is a split extension containing~$Q_8$
as a normal subgroup, 
$G_{24}\iso C_3\ltimes Q_8$. As a subgroup 
of the stabilizer group~$\mathbb{G}_2$ 
this is generated by~$i,j,\omega$, and 
by~\eqref{eq:Q8explicit},
\[
\omega i\omega^{-1}=j,
\quad
\omega j\omega^{-1}=k,
\quad
\omega k\omega^{-1}=i.
\]
Here we identify $C_3$ with the subgroup 
generated by~$\omega$.

The right action of $G_{24}$ on 
$(K_2)_*(S^0\cup_\nu e^4\cup_\eta e^6)$ 
in terms of the generators 
$z_k=u^k\bar{z}_{k}\in(K_2)_{2k}(S^0\cup_\nu e^4\cup_\eta e^6)$
inherited from $BP_*(S^0\cup_\nu e^4\cup_\eta e^6)$ 
can be deduced using~\eqref{eq:DbleUpside?}:
\begin{equation}\label{eq:G24rightaction}
\left\{
\begin{aligned}
z_0\cdot i &= z_0, & z_0\cdot j &= z_0, & z_0\cdot \omega &= z_0, \\
z_4\cdot i &= z_4+u^2z_0, & z_4\cdot j &= z_4+\omega u^2z_0, & z_4\cdot\omega &= z_4, \\
z_6\cdot i &= z_6+ uz_4 + \omega u^3z_0, & z_6\cdot j &= z_6 + \omega^2 uz_4 +\omega u^3z_0, & z_6\cdot\omega &= z_6. \\
\end{aligned}
\right.
\end{equation}

Using Brauer characters it is routine to 
verify that $\F_4[G_{24}]$ has~$3$ simple 
modules each of which is $1$-dimensional 
with $8$-dimensional projective cover. The 
summands in the corresponding decomposition 
of the module 
$(K_2)_*(S^0\cup_\nu e^4\cup_\eta e^6)$ is 
obtained from the subspace of $C_3$-invariants 
on multiplying them by~$1,u,u^2$. 

The $C_3$-invariants in 
$(K_2)_*(S^0\cup_\nu e^4\cup_\eta e^6)$ 
is isomorphic to $(K(2)\F_4)_*(S^0\cup_\nu e^4\cup_\eta e^6)$,
i.e., the original Morava $K$-theory with 
coefficients in $\F_4$:
\[
(K(2)\F_4)_*(-) = \F_4\otimes_{\F_2} K(2)_*(-).
\]
This is of course a module over the graded field
$(K(2)\F_4)_*=\F_4[v_2,v_2^{-1}]=\F_4[u^3,u^{-3}]$.
Furthermore it has an action of the skew 
Hecke algebra 
$(K_2)_*^{C_3}\{C_3\backslash G_{24}/C_3\}\iso(K(2)\F_4)_*\{Q_8\}$
discussed in Appendix~\ref{app:CrossProd}.

Remembering that the right action of $\omega$ 
on $u^k$ satisfies $u^k\mapsto\omega^{-k}u^k=\omega^{2k}u^k$,
we find that the following $8$ elements form 
a $(K(2)\F_4)_*$-basis for
$(K_2)_*^{C_3}\{C_3\backslash G_{24}/C_3\}\iso(K(2)\F_4)_*\{Q_8\}$:
\begin{align*}
1H,\; & i^2H=j^2H=k^2H, \\
iH+jH+kH,\;& i^3H+j^3H+k^3H, \\
u(iH+\omega^2jH+\omega kH),\; & u(i^3H+\omega^2j^3H+\omega k^3H),\; \\
u^{-1}(iH+\omega jH+\omega^2kH),\; &
u^{-1}(i^3H+\omega j^3H+\omega^2k^3H).
\end{align*}
Their actions on
$(K(2)\F_4)_*(S^0\cup_\nu e^4\cup_\eta e^6)$ 
have the following matrices with respect to 
the basis $z_0,z_4,z_6$:
\begin{align*}
\begin{bmatrix}
1 & 0 & 0 \\
0 & 1 & 0 \\
0 & 0 & 1 
\end{bmatrix},\; &
\begin{bmatrix}
1 & 0 & u^3 \\
0 & 1 & 0 \\
0 & 0 & 1 
\end{bmatrix},  \\
\begin{bmatrix}
1 & 0 & \omega u^3 \\
0 & 1 & 0 \\
0 & 0 & 1 
\end{bmatrix},\; & 
\begin{bmatrix}
1 & 0 & \omega^2u^3 \\
0 & 1 & 0 \\
0 & 0 & 1 
\end{bmatrix},  \\
\begin{bmatrix}
0 & u^3 & 0 \\
0 & 0 & 0 \\
0 & 0 & 0 
\end{bmatrix},\; &
\begin{bmatrix}
0 & u^3 & 0 \\
0 & 0 & 0 \\
0 & 0 & 0 
\end{bmatrix},  \\
\begin{bmatrix}
0 & 0 & 0 \\
0 & 0 & 1 \\
0 & 0 & 0 
\end{bmatrix},\;  &
\begin{bmatrix}
0 & 0 & 0 \\
0 & 0 & 1 \\
0 & 0 & 0 
\end{bmatrix}. 
\end{align*}

\bigskip
\section*{Concluding remarks}
The main import of this paper is the
appearance of unexpected relationships
between seemingly disparate topics. It 
has long been noted that there appear 
to be connections between the cohomology 
of some of the~$\StA(n)$ and that of 
finite groups. The case of~$Q_8$ is 
one where such connections have been 
observed and we provide further evidence 
of this. However, it is unclear whether 
there are other examples, perhaps in 
higher chromatic heights.


\bigskip\hrule\bigskip

\appendix

\section{Skew group rings and their modules}
\label{app:CrossProd}

The results in this appendix are aimed at
the specific circumstances that occur in
chromatic homotopy theory. More general
statements on skew (or twisted) group
rings can be found in
Passman~\cite{DSP:InfCrossProd}*{section~4};
Lam~\cite{TYL:NonCommRings}*{chapter~7} is 
a good source on local and semilocal rings. 
We also discuss skew Hecke algebras which 
do not seem to be extensively documented.

\subsection*{Skew group rings}
We begin with a result including 
both~\cite{TYL:NonCommRings}*{theorem~20.6}
and~\cite{DSP:InfCrossProd}*{theorem~4.2}
as special cases.

Recall from Lam~\cite{TYL:NonCommRings}*{\S20}
that a ring $A$ is \emph{semilocal} if $A/\rad A$
is semisimple, where~$\rad A$ is the Jacobson 
radical of~$A$.
\begin{prop}\label{prop:RingThyResult}
Suppose that $A\subseteq B$ is a semilocal
subring where~$B$ is finitely generated as
a left $A$-module. Let $\mathfrak{a}\lhd A$
be a radical ideal, $\mathfrak{b}=\rad B\lhd B$
and $B\mathfrak{a}\subseteq\mathfrak{a}B$.
Then \\
\emph{(a)}
$\mathfrak{a}\subseteq\mathfrak{b}$;  \\
\emph{(b)}
$B$ is semilocal;  \\
\emph{(c)} There is a $k\geq1$ such that
$\mathfrak{b}^k\subseteq B\mathfrak{a}$.
\end{prop}

\begin{proof}
(a) Let $M$ be a simple left $B$-module.
Then $M$ is cyclic and so is finitely
generated over $B$ and therefore over $A$.
Also,
\[
B(\mathfrak{a}M)=(B\mathfrak{a})M
\subseteq(\mathfrak{a}B)=\mathfrak{a}(BM)
\mathfrak{a}M,
\]
so $\mathfrak{a}M\subseteq M$ is a $B$-submodule.
If $\mathfrak{a}M\neq0$ then the $A$-module
$M$ satisfies $\mathfrak{a}M=M$, so by
Nakayama's Lemma, $M=0$. So we must have
$\mathfrak{a}M=0$.

Since $\mathfrak{a}$ annihilates every simple
$B$-module, $\mathfrak{a}\subseteq\mathfrak{b}$.   \\
(b) The finitely generated left $A/\mathfrak{a}$-module
$B/B\mathfrak{a}B$ is also a ring which is
left Artinian with radical $\mathfrak{b}/B\mathfrak{a}B$.
This implies that the quotient ring $B/\mathfrak{b}$
is left semisimple. \\
(c) The finitely generated left
$A/\mathfrak{a}$-module $B/B\mathfrak{a}$
is left Artinian. The $B$-submodules
$\mathfrak{b}^k/B\mathfrak{a}$ form a
decreasing chain which must stabilize,
so for some $k\geq1$,
\[
\mathfrak{b}^k/B\mathfrak{a}
=
\mathfrak{b}^{k+1}/B\mathfrak{a}
=
\mathfrak{b}(\mathfrak{b}^k/B\mathfrak{a}),
\]
By Nakayama's Lemma $\mathfrak{b}^k/B\mathfrak{a}B=0$,
hence $\mathfrak{b}^k\subseteq B\mathfrak{a}$.
\end{proof}

The special case $\mathfrak{a}=\rad A$
is particularly important. In practise
we will consider the case where
$B\mathfrak{a}=\mathfrak{a}B$
so~$B\mathfrak{a}\lhd B$. This is true
when $R$ is a ring with a group acting 
on it by automorphisms; then the radical 
$\mathfrak{r}\lhd R$ is necessarily invariant 
so we can apply our results with $A=R$ 
and $B=R\langle G\rangle$, the skew group 
ring. This recovers~\cite{DSP:InfCrossProd}*{theorem~4.2}.
We will discuss this special case in
detail, making additional assumptions
relevant in chromatic stable homotopy
theory.

Let $(R,\mathfrak{m})$ be a complete and
Hausdorff (i.e., $\bigcap_{r\geq1}\mathfrak{m}^r=0$)
Noetherian commutative local ring with
residue field $\kappa=R/\mathfrak{m}$ of
positive characteristic~$p$. Let~$G$ be
a finite group which acts on~$R$ by
(necessarily local) automorphisms, so
that~$G$ also acts on~$\kappa$ by field
automorphisms.

We can form the \emph{skew group rings}
$R\langle G\rangle$ and $\kappa\langle G\rangle$;
if the action of~$G$ on~$R$ or~$\kappa$ is
trivial then we have the ordinary group
ring~$R[G]$ or~$\kappa[G]$. The subset
\[
\mathfrak{M}
= R\langle G\rangle\mathfrak{m}
= \mathfrak{m}R\langle G\rangle
= \{\sum_{g\in G}x_g g : x_g\in\mathfrak{m}\}
\subseteq R\langle G\rangle
\]
is a two-sided ideal with quotient ring
$R\langle G\rangle/\mathfrak{M}\iso\kappa\langle G\rangle$.
There is a maximal ideal
\[
\mathfrak{n}=\{\sum_{g\in G}y_g(g-1) : y_g\in\kappa\}
\lhd\kappa\langle G\rangle
\]
with quotient ring
$\kappa\langle G\rangle/\mathfrak{n}\iso\kappa$
defining the trivial $\kappa\langle G\rangle$-module
as well as the trivial $R\langle G\rangle$-module
$R\langle G\rangle/\mathfrak{N}$ where
\[
\mathfrak{N} =
\mathfrak{M} + \{\sum_{g\in G}z_g(g-1) : z_g\in R\}
\lhd R\langle G\rangle.
\]

Our next two results follow from our
Proposition~\ref{prop:RingThyResult} as
well as being special cases
of~\cite{DSP:InfCrossProd}*{theorem~4.2}.
\begin{lem}\label{lem:R<G>-semilocal}{\ph{a}} \\
\emph{(a)}
$\kappa\langle G\rangle$ is semilocal; \\
\emph{(b)}
The ideal $\mathfrak{M}\lhd R\langle G\rangle$
is a radical ideal and $R\langle G\rangle$ is
semilocal. \\
\emph{(c)} The simple $R\langle G\rangle$-modules
are obtained by pulling back the simple modules
of $\kappa\langle G\rangle$ along the quotient
homomorphism $R\langle G\rangle\to\kappa\langle G\rangle$.
\end{lem}
\begin{proof}
(a) This follows from Artin-Wedderburn theory
since $\kappa\langle G\rangle$ is a finite
dimensional $\kappa$-vector space and hence
Artinian. \\
(b) Use Proposition~\ref{prop:RingThyResult}. \\
(c) This follows from (b).
\end{proof}

A detailed discussion of lifting of idempotents 
and results on Krull-Schmidt decompositions 
for complete local Noetherian rings can be 
found in Lam~\cite{TYL:NonCommRings}*{section~21}.

Now we can deduce an important special case.
\begin{lem}\label{lem:R<G>-local}
Suppose that $G$ is a $p$-group. Then \\
\emph{(a)}
$\kappa\langle G\rangle$ is local with
unique maximal left/right ideal\/
$\mathfrak{n}$ equal to the radical\/
$\rad\kappa\langle G\rangle$; \\
\emph{(b)}
$R\langle G\rangle$ is local with unique
maximal left/right ideal\/ $\mathfrak{N}$.

Hence $R\langle G\rangle$ and $\kappa\langle G\rangle$
each have the unique simple module~$\kappa$.
\end{lem}
\begin{proof}
(a) Suppose that $S$ is a (non-trivial)
simple left $\kappa\langle G\rangle$-module.
For $0\neq s\in S$, consider the finite
dimensional $\F_p$-subspace $\F_p[G]s\subseteq S$
whose cardinality is a power of~$p$. It is
also a non-trivial finite $\F_p[G]$-module,
so the $p$-group~$G$ acts linearly with~$0$
as a fixed point. Since every orbit has
cardinality equal to a power of~$p$ there
must be at least one other fixed point
$v\neq0$ and this spans a $\kappa\langle G\rangle$-submodule
$\kappa v\subseteq S$. It follows that
$S=\kappa v\iso\kappa$. Of course if
the $G$-action on $\kappa$ is trivial,
$\kappa\langle G\rangle=\kappa[G]$ and
this argument is well-known. \\
(b) This is immediate from (a) together
with parts (b) and (c) of
Lemma~\ref{lem:R<G>-semilocal}.
\end{proof}
\begin{cor}\label{cor:R<G>-local}
If $G$ is a $p$-group, then
$\mathfrak{N}\lhd R\langle G\rangle$
is the unique maximal ideal and $R\langle G\rangle$
is $\mathfrak{N}$-adically complete and Hausdorff.
\end{cor}
\begin{proof}
This follows from Proposition~\ref{prop:RingThyResult}(c):
some power of $\mathfrak{N}$ is contained
in $\mathfrak{M}=R\langle G\rangle\mathfrak{m}$,
and for $k\geq1$,
$\mathfrak{M}^k = R\langle G\rangle\mathfrak{m}^k\subseteq\mathfrak{N}^k$.
Therefore the $\mathfrak{N}$-adic, $\mathfrak{M}$-adic
and $\mathfrak{m}$-adic topologies agree.
\end{proof}

We recall that for a local ring, every
projective module is free by a theorem 
of Kaplansky~\cite{IK:ProjMods}*{theorem~2}, 
so in statements involving local rings, 
projective modules can be taken to be 
free.
\begin{lem}\label{lem:ProjMods}{\ } \\
\emph{(a)} Let $P$ be a projective 
$R\langle G\rangle$-module. Then~$P$ is 
a projective $R$-module. \\
\emph{(b)}
Let $Q$ be a finitely generated projective
$\kappa\langle G\rangle$-module. Then there
is a projective $R\langle G\rangle$-module
$\tilde{Q}$ for which
$\kappa\langle G\rangle\otimes_{R\langle G\rangle}\tilde{Q}\iso Q$.
\end{lem}
\begin{proof}
(a) Every projective module is a retract of
a free module and $R\langle G\rangle$-module
is a free $R$-module. \\
(b) By the Krull-Schmidt theorem, we may
express $Q$ as a coproduct of projective
indecomposable $\kappa\langle G\rangle$-modules,
so it suffices to assume $Q$ is a projective
indecomposable, hence cyclic. Viewing $Q$
as an $R\langle G\rangle$-module we can
choose a cyclic projective module $\tilde{Q}$
with an epimorphism $\pi\:\tilde{Q}\to Q$.
\end{proof}

We will make use of the following result.
\begin{lem}\label{lemLiftingEndoTriv}
Suppose that $M$ is an $R\langle G\rangle$-module
which is finitely generated free as an $R$-module.
If $\kappa\otimes_R M$ is an endotrivial
$\kappa\langle G\rangle$-module, then~$M$
is an endotrivial $R\langle G\rangle$-module.
\end{lem}
\begin{proof}
Let $\End_R(M)=\Hom_R(M,M)$ with its usual
left $R\langle G\rangle$-module structure.
If $\kappa\otimes_R M$ is endotrivial then
as $\kappa\langle G\rangle$-modules,
\[
\kappa\otimes_R\End_R(M)
\iso
\End_\kappa(\kappa\otimes_R M,\kappa\otimes_R M)
\iso\kappa\oplus P
\]
where $P$ is a projective $\kappa\langle G\rangle$-module.
Recall that the units give monomorphisms
$R\to\End_R(M)$ and $\kappa\to\kappa\otimes_R\End_R(M)$,
where the latter is split.

Now choose a projective $R\langle G\rangle$-module
$\tilde{P}$ with an epimorphism $\pi\:\tilde{P}\to P$
and $\kappa\otimes_{R}\tilde{P}\iso P$. There is
a commutative diagram of solid arrows with exact
rows
%
\[
\begin{tikzcd}
0\ar[r] & R\ar[r]\ar[d,twoheadrightarrow] & \End_R(M)\ar[r]\ar[d,twoheadrightarrow] 
& \tilde{P}\ar[d, twoheadrightarrow, "\pi"]\ar[dl, dashed, "{\pi'}"']\ar[l, bend right=25,"{\pi''}"', dotted] & \\
0\ar[r] & \kappa\ar[r] & \kappa\otimes_R\End_R(M)\ar[r] 
& \ar[r]P\ar[l, dashed,bend left=25, "\sigma"] & 0 \\
\end{tikzcd}
\]
and the composition $\sigma\circ\pi$ lifts
to $\pi''\:\tilde{P}\to\End_R(M)$. On applying
$\kappa\otimes_R(-)$ to the composition
\[
\tilde{P}\xrightarrow{\pi''}\End_R(M)\xrightarrow{\pi}P
\]
we obtain the composition
\[
P\xrightarrow{\sigma}\kappa\otimes_R\End_R(M)\to P
\]
which is an epimorphism. Using Nakayama's Lemma
we now see that $\End_R(M)\to\tilde{P}$ is an
epimorphism, hence $\End_R(M)\iso R\oplus\tilde{P}$
and so $M$ is endotrivial.
\end{proof}

Although we don't really make use of this
fact, we remark that an appropriate dual
of a skew group ring over a commutative 
ring admits the structure of a Hopf algebroid. 
A generalisation to skew Hecke algebras 
appears in the appendix of~\cite{AB:LandweberFitnThm}
where we referred to them as twisted 
Hecke algebras.

%

\subsection*{Skew Hecke algebras}

Hecke algebras are commonly encountered 
in the study of modular forms and 
representation theory and they also 
appear as stable operations in elliptic 
cohomology and topological modular forms.
A general algebraic introduction can be 
found in Krieg~\cite{AK:HeckeAlgebras}. 
Here we describe a skew version, for 
a recent algebraic account see Waldron 
\& Loveridge~\cite{JW&LDL:SkewHeckeAlg}.

To simplify things we will assume that 
$G$ is a \emph{finite} group acting on 
a commutative $\k$-algebra~$A$ by algebra
automorphisms. If $H\leq G$ we may form 
the skew group algebra $A\langle G\rangle$.
We will indicate the action of $g\in G$ 
on $a\in A$ by writing $\leftidx{^g}{a}{}$.

The free left $A$-module $A\,G/H$ is also
a left $A\langle G\rangle$-module and we 
may define a \emph{skew Hecke algebra} 
by 
\[
A^H\{H\backslash G/H\} = \End_{A\langle G\rangle}(A\{G/H\})^\op
= \Hom_{A\langle G\rangle}(A\{G/H\},A\{G/H\})^\op,
\]
the opposite of the endomorphism algebra of 
the $A\langle G\rangle$-module~$A\{G/H\}$. 

By standard adjunction results, there are 
isomorphisms of $\k$-modules
\begin{align*}
A^H\{H\backslash G/H\}
&\iso \Hom_{A\langle G\rangle}(A\{G/H\},A\{G/H\})^\op  \\
&\iso \Hom_{\k[G]}(\k\{G/H\},A\{G/H\}) \\
&\iso \Hom_{\k[H]}(\k,A\{G/H\}),
\end{align*}
which actually isomorphisms of 
$\leftidx{^H}{A}{}$-modules. The last term 
can be identified with the $H$-fixed point 
set
\[
\leftidx{^H}{(A\{G/H\})}{} =
\biggl\{\sum_{x\: G/H}r_x\, xH : 
\forall x,\forall h\in H,\; \leftidx{^h}{r_x}{} = r_{hx}\biggr\}.
\]
Here we adopt notation from~\cite{AK:HeckeAlgebras}: 
$\ds\sum_{x\: G/H}$ indicates summation over 
a complete set of coset representatives~$x$ 
for~$G/H$. If the $G$-action on~$A$ is trivial,
$\leftidx{^H}{(A\{G/H\})}{}$ is the free $A$-module 
on the set of double cosets $H\backslash G/H$ which 
agrees with the classical notion of Hecke algebra.
Of course we can view $A^H\{H\backslash G/H\}$ as
an $\leftidx{^H}{A}{}$-algebra where the unity 
comes from the double coset~$H1H$ and is the 
element~$1H\in A\{G/H\}$.

To make the multiplication $*$ on $A^H\{H\backslash G/H\}$ 
explicit, we identify~$\alpha\in\End_{A\langle G\rangle}(A\{G/H\})^\op$
with the corresponding element of $\leftidx{^H}{(A\{G/H\})}{}$,
\[
\alpha(1H) = \sum_{x\:G/H}a_x\,xH
\]
where $a_x\in A$. Then for $\beta\in\End_{A\langle G\rangle}(A\{G/H\})^\op$
with 
\[
\beta(1H) = \sum_{x\:G/H}b_x\,xH
\]
we obtain 
\begin{equation}\label{eq:HeckeProduct}
\alpha * \beta = \sum_{x,y\:G/H}a_x \leftidx{^x}{b}{_y}\,(xy)H
= \sum_{x,y\:G/H}a_x \leftidx{^x}{b}{_{x^{-1}y}}\,yH.
\end{equation}

Now for a left $A\langle G\rangle$-module~$M$, 
its $H$-fixed point set 
\[
\leftidx{^H}{M}{} 
\iso \Hom_{\k[H]}(\k,M)
\iso \Hom_{A\langle G\rangle}(A\{G/H\},M)
\]
is naturally a \emph{right} 
$\End_{A\langle G\rangle}(A\{G/H\})$-module
and therefore a \emph{left} $A^H\{H\backslash G/H\}$-module. 
For $m\in\leftidx{^H}{M}{}$ and $\alpha\in A^H\{H\backslash G/H\}$ 
the action is given by 
\begin{equation}\label{eq:FixedPtAction}
\alpha*m = \sum_{x\:G/H}a_x\,\leftidx{^x}{m}{}.
\end{equation}
When $H\lhd G$, as sets $H\backslash G/H=G/H$ 
and
\[
A^H\{H\backslash G/H\} 
\iso (\leftidx{^H}{A}{})\langle G/H\rangle.
\]

A more interesting situation that we encounter 
in Section~\ref{sec:G24action} involves a 
semidirect product~$G=HN\iso H\ltimes N$. 
Each double coset in $H\backslash G/H$ has 
the form~$HnH$ where $n\in N$ is uniquely 
determined up to $H$-conjugacy. So as left 
$\leftidx{^H}{A}{}$-modules,
\[
A^H\{H\backslash G/H\} \iso A^H\{N\}.
\] 
For a left $A\langle G\rangle$-module~$M$, 
the action of the element corresponding 
to $n\in N$ on $\leftidx{^H}{M}{}$ is 
given by
\begin{equation}\label{eq:DbleCosetAction}
nH*m = \sum_{h\: H/C_H(n)} hnh^{-1}m,
\end{equation}
where the sum is really taken over the 
set of $H$-conjugates of~$n$.


\end{document}